\documentclass{article}
\usepackage{graphicx,amsfonts,amsmath,mathrsfs,amssymb,amsthm,url,color,tikz}
 \usepackage{fancyhdr,indentfirst,bm,enumerate,natbib, float,subfigure}
  \usepackage[colorlinks=true,citecolor=blue]{hyperref}
   \usepackage{dsfont}
\usetikzlibrary{arrows.meta,calc}

   \newcommand{\E}{\mathbb{E}}
    \newcommand{\R}{\mathbb{R}}
     
      \renewcommand{\P}{\mathbb{P}}
  \newcommand{\id}{\mathds{1}}

\newtheorem{theorem}{Theorem}[section]
 
  \newtheorem{lemma}[theorem]{Lemma}
   \newtheorem{proposition}[theorem]{Proposition}
    \newtheorem{definition}{Definition}[section]

        \newtheorem{remark}[theorem]{Remark}

\numberwithin{equation}{section}
 \numberwithin{theorem}{section}
  \numberwithin{table}{section}

\renewcommand{\cite}{\citet}
 \allowdisplaybreaks
  
\begin{document}

\title{Full negative regression dependence: counterexamples for tail dependence and negative association}

\author{
 Yuting Su\thanks{School of Management, University of Science and Technology of China,  China.
 Email: \url{syt20020224@mail.ustc.edu.cn}}
 \and  Taizhong Hu\thanks{School of Management, University of Science and Technology of China,  China. 
 Email: \url{thu@ustc.edu.cn}}      	}

\date{\today}

\maketitle

\begin{abstract}

We show that full negative regression dependence implies neither negative right-tail dependence nor negative left-tail dependence.  A strictly positive distribution on a $2\times2\times3$ rectangular state space provides an exact
counterexample for the right-tail case, and its coordinatewise reflection yields a counterexample for the left-tail case. We further prove that this state-space size is minimal. The counterexample is established via a finite upper-set criterion using exact integer arithmetic. We also resolve the open question of whether full negative regression dependence implies negative association.  A strictly positive distribution on $\{0,1,2\}^4$ satisfies full negative regression dependence yet violates negative association; this  counterexample is verified by $17\,532$ exact integer inequalities. The main results of this paper were first uncovered by the multi-agent mathematical system Eureka and later verified by the authors.
\medskip
		
\noindent \textbf{MSC2000 subject classification}: Primary 60E15; Secondary 62H05.

\end{abstract}

\section{Introduction}

Negative dependence encompasses a key family of probabilistic properties that characterize inverse stochastic relationships between the components of a random vector. Such dependence structures arise naturally in a wide range of research areas, including reliability theory, occupancy problems, balls-and-bins models, order statistics, tournament models, risk management, and operations research; we refer readers to \cite{Leh66}, \cite{AS81}, \cite{JP83}, \cite{BSS82,BSS85}, \cite{BHJK96}, \cite{DR98}, \cite{Pem00}, \cite{HX06}, \cite{PW15}, \cite{MR23}, \cite{KLW24}, \cite{SZH26}, and \cite{ZSH26} for relevant literature and for various concepts of negative dependence. 

Well-established negative dependence concepts include \emph{negative association} (NA) and \emph{full negative regression dependence} (full NRD), as well as their tail-based extensions: \emph{negative right-tail dependence} (NRTD) and \emph{negative left-tail dependence} (NLTD). Specifically, NA enforces a global covariance condition, requiring that increasing functions defined on disjoint coordinate subsets yield non-positive covariance. In contrast, full NRD stipulates that the conditional distribution of one coordinate block decreases in the standard stochastic order when a disjoint coordinate block takes larger fixed values. Differing from NRD, NRTD and NLTD impose upper-tail and lower-tail conditional constraints, respectively, instead of exact point-wise conditioning. Rigorous definitions of NA, NRD, NLTD, and NRTD are presented in Section \ref{sec:definitions}. 
 
It is known that NA does not imply NRD, NLTD or NRTD \citep[][Example 1]{SZH26}, and that NRTD does not imply NRD or NLTD \citep[][Example 4]{SZH26}. The converse direction, however, has remained open for many years. In their foundational work on balls‑and‑bins models, \cite{DR98} claimed that full NRD implies both NRTD and NLTD, and they further raised the open question whether full NRD necessarily yields NA; see Section~5 therein. Later, \cite{HX06} identified a logical gap in the original argument for the tail‑dependence implications. The missing step concerns preservation of regression monotonicity after conditioning on tail events. Subsequent literature has listed these as open conjectures. \cite{SZH26} isolated the unproven intermediate assertion, and \cite{ZSH26} listed both implications as open problems in their survey‑style discussion of negative‑dependence constructions.

This paper answers both families of open questions in the negative. We construct explicit strictly positive distributions on finite rectangular state spaces serving as counterexamples. First, we exhibit a distribution on a $2\times 2\times 3$ state space satisfying full NRD but violating NRTD. Coordinate‑wise reflection yields a counterexample against NLTD. We further establish minimality: any counterexample witnessing the failure of these tail implications must live on a projected rectangular state space of size at least $12$, and our $2\times 2\times 3$ example achieves this lower bound. The underlying mechanism behind the tail‑order reversal is a finite‑state analogue of Simpson’s paradox: regression monotonicity holds within each fixed inner layer, but mixing across layers under tail‑event conditioning destroys the required stochastic order.

Second, we answer the long‑standing open problem posed in \cite{DR98}: full NRD does \emph{not} imply NA. We provide a strictly positive distribution on $\{0,1,2\}^4$ which satisfies full NRD but produces a strictly positive covariance between two increasing functions supported on disjoint coordinate blocks. Verification requires checking $17\,532$ exact integer inequalities. We also derive boundary results describing when full NRD does force NA: every three‑dimensional full‑NRD random vector is automatically NA, and four‑dimensional strictly‑positive full‑NRD distributions with at least one binary coordinate are also NA. Consequently, any four‑dimensional strictly‑positive counterexample must have at least three states per coordinate, making $\{0,1,2\}^4$ coordinate‑wise minimal among rectangular state spaces.

All counterexample validations rely on a finite upper‑set criterion: for integer‑valued random vectors, full NRD reduces to a finite set of cross‑product integer inequalities, avoiding floating‑point arithmetic entirely. The mathematical conjectures solved in this paper were first discovered by the multi‑agent mathematical system Eureka, and all claims are subsequently verified analytically and by independent computer implementations. Supplementary code is provided to ensure computational reproducibility.

The remainder of this paper is organized as follows. Section~\ref{sec:definitions} compiles the required definitions, notation, and the fundamental finite‑state criterion for full NRD. Section~\ref{sec:tailcounterexample} discusses implications for tail‑dependence, presenting the $2\times2\times3$ counterexample, its minimality proof, and the reflected left‑tail counterexample. Section~\ref{sec:counterexample-na} investigates whether full NRD implies NA, and presents the four‑dimensional counterexample along with propositions concerning  dimensional and state‑size bounds.

\section{Definitions and conventions}
\label{sec:definitions}

Recall that a random vector $\bm X=(X_1, \ldots, X_n)$ is said to be smaller than another random vector $\bm Y=(Y_1, \ldots, Y_n)$ in the {\it usual stochastic order}, denoted by $\bm X\le_{\rm st} \bm Y$, if 
\begin{equation}
 \label{eq:st}
    \E [\varphi (\bm X)]\le \E [\varphi(\bm Y)]
\end{equation}
holds for all increasing functions $\varphi$ for which the expectations exist \citep[][Section 4B]{SS07}. On a finite state space, it is equivalent to test the inequality \eqref{eq:st} for indicator functions of upper sets. 

Throughout, we denote by $[\bm X|A]$ any random vector/variable whose distribution is the conditional distribution of $\bm X$ given event $A$. For any $\bm x\in \R^n$ and $J\subset [n]:=\{1, 2, \ldots, n\}$, let $\{X_j, j\in J\}$, $\{X_j\le x_j, j\in J\}$ and $\{X_j>x_j, j\in J\}$ be abbreviated by $\bm X_J$, $\bm X_J\le \bm x_J$ and $\bm X_J> \bm x_J$, respectively.  Conditional distributions are compared only when the conditioning events have positive probability. 
\begin{definition}
{\rm \citep{AS81,JP83}}\ \ Let $\bm X=(X_1, \ldots, X_n)$ be a random vector. $\bm X$ is said to be negatively associated {\rm (NA}) if for every pair of disjoint subsets $A_1, A_2\subset [n]$,
\begin{equation}
 \label{eq:na}
     {\rm Cov} (\psi_1(\bm X_{A_1}), \psi_2(\bm X_{A_2})) \le 0
\end{equation}
whenever $\psi_1$ and $\psi_2$ are coordinate-wise increasing such that the covariance exists.
\end{definition}

\begin{definition}
  \label{def-nrd}
{\rm \citep{DR98}}\ \ Let ${\bm X}=(X_1, \ldots, X_n)$ be a random vector. $\bm X$ is said to be   \vspace*{-5pt}

\begin{itemize}
\item[{\rm (1)}] negatively regression dependent {\rm (NRD)} if
   \begin{equation} 
    \label{eq:nrd}
     \left [\bm X_I| \bm X_J=\bm x_J\right ]\ge_{\rm st} \left [\bm X_I| \bm X_J =\bm x_J^\ast\right ],
   \end{equation}
   where $\bm x_J\le \bm x_J^\ast$, and $I$ and $J$ are any disjoint subsets of $[n]$;

\item[{\rm (2)}] negatively left-tail dependent {\rm (NLTD)} if \eqref{eq:nrd} is replaced by
  \begin{equation} 
     \label{eq:nltd}
    \left [\bm X_I|\bm X_J\le\bm x_J\right ]\ge_{\rm st} \left [\bm X_I|\bm X_J\le\bm x_J^\ast\right ];
  \end{equation}

\item[{\rm (3)}] negatively right-tail dependent {\rm (NRTD)} if \eqref{eq:nrd} is replaced by
  \begin{equation} \label{eq:nrtd}
    \left [\bm X_I|\bm X_J>\bm x_J\right ]\ge_{\rm st}\left [\bm X_I| \bm X_J >\bm x_J^\ast\right ].
  \end{equation}
\end{itemize}
\end{definition}

We use \emph{full} NRD to emphasize that \eqref{eq:nrd} is required for arbitrary disjoint blocks, rather than only for $|J|=1$. These conventions agree with \cite[Definition~1.1]{HX06}; see also the
recent formulations in \cite{SZH26} and \cite{ZSH26}. It should be pointed out that, via a limiting argument, the symbols ``$\le$'' and ``$>$'' in \eqref{eq:nltd} and \eqref{eq:nrtd} can be replaced by ``$<$'' and ``$\ge$'', respectively. 

Coordinate projections preserve full NRD, since every comparison for a projected vector is already one of the comparisons in \eqref{eq:nrd}.

For integer-valued random variables, the following elementary criterion turns full NRD into a finite list of integer inequalities.

\begin{lemma}
 \label{lem:certificate}
Consider a mass function defined on a finite product of chains, with probabilities proportional to positive integer weights $w(\bm x)$.  Fix disjoint nonempty $I,J$, an upper set $U$ in the state space of $\bm X_I$, and a value $\bm a$ of $\bm X_J$.  Define
$$
  w(U,\bm a)=\sum_{\substack{\bm x_I\in U\\ \bm x_J=\bm a}} w(\bm x),   \qquad
    w(\bm a) =\sum_{\bm x_J=\bm a} w(\bm x).
$$
Then the corresponding distribution is full {\rm NRD} if and only if
\begin{equation}\label{eq:slack}
  \Delta(I,J,U,\bm a,\bm b) := w(U,\bm a) w(\bm b) -w(U,\bm b) w(\bm a)\ge 0
\end{equation}
for every such $I,J$, every nontrivial upper set $U$, and every distinct comparable pair $\bm a < \bm b$.
\end{lemma}

\begin{proof}
The denominators $w(\bm a)$ and $w(\bm b)$ are positive.  Dividing \eqref{eq:slack} by their product gives
$$
  \P(\bm X_I\in U\mid \bm X_J=\bm a)   \ge \P(\bm X_I\in U\mid \bm X_J=\bm b).
$$
By the definition of stochastic order, these inequalities for all upper sets are equivalent to the stochastic-order comparison \eqref{eq:nrd}.  Empty and full upper sets give equality and may be omitted.
\end{proof}

\section{Full NRD implies neither NRTD nor NLTD}
\label{sec:tailcounterexample}

\subsection{State spaces on which failure is impossible}

We first recall the classical one-coordinate mixing argument; see \cite[Chapter~5]{BP81}.  A proof is included because the distinction between one and two conditioning coordinates is central to the minimality statement.

\begin{proposition}
 \label{prop:singleton-tail}
Suppose that $\bm X$ is full {\rm NRD}.  Then \eqref{eq:nltd} and \eqref{eq:nrtd} hold whenever the conditioning block $J$ is a singleton, without any restriction on the number of states of that coordinate.
\end{proposition}

\begin{proof}
Write $J=\{j\}$, and fix an increasing function $f$ of $\bm X_I$.  Full NRD implies that $ m(t):=\E [f(\bm X_I)\,|\, X_j=t]$ is decreasing in $t$.  If $s\le t$, then $[X_j\mid X_j>s] \le_{\rm st}[X_j\mid X_j>t]$. Hence,
\begin{align*}
  \E [f(\bm X_I)\,|\, X_j>s] & = \E\Big \{\E [f(\bm X_I)\,|\, X_j] \mid X_j>s\Big \}
  =\E [m(X_j) |X_j>s]  \\
  & \ge \E [m(X_j)\, |\, X_j>t]=\E [f(\bm X_I)\,|\, X_j>t].
\end{align*}
This is NRTD for $J=\{j\}$.  Likewise, $[X_j\mid X_j\le s] \le_{\rm st}[X_j\mid X_j\le t]$ for $s<t$, so the same argument proves NLTD.  
\end{proof}

The following result gives the key structural reason that no counterexample can exist on a binary rectangular state space.

\begin{proposition}\label{thm:binary-block}
Let $\bm X$ be full {\rm NRD}, and let $I,J$ be disjoint index sets.  If every $X_j$, $j\in J$, has at most two states, then both inequalities \eqref{eq:nltd} and \eqref{eq:nrtd} hold for the pair $(I,J)$.
\end{proposition}

\begin{proof}
Singleton coordinates produce either trivial constraints or null events under tail conditioning and may therefore be dropped from $J$.  Order-preserving changes of scale then allow us to take the state space of each remaining $X_j$, $j\in J$, to be $\{0,1\}$.  Every positive-probability right-tail event has the form $\{X_j=1:j\in S\}$ for some $S\subseteq J$; the other coordinates of $J$ are unconstrained.  If the threshold vector is increased, the set $S$ can only grow.

Fix $S\subseteq J$ and $j\in J\setminus S$.  Let $\mu_0$ and $\mu_1$ be the conditional distributions of $\bm X_I$ given $\bm X_S=\id$ and $X_j=0$ or $1$, respectively, whenever both are well-defined.  Full NRD, applied to the exact conditioning block $S\cup\{j\}$, gives
$$
  \mu_1\le_{\rm st} \mu_0.
$$
The distribution of $\bm X_I$ given $\bm X_S=\id$ is a mixture
$$
  (1-p) \mu_0 + p \mu_1,   \qquad p=\P (X_j=1\,|\, X_S=\id).
$$
It therefore stochastically dominates $\mu_1$, the distribution obtained after imposing the additional condition $X_j=1$.  If one of the two exact events has probability zero, the same conclusion is either equality or the stronger event is inadmissible.  Adding the coordinates of a larger set one at a time proves NRTD.

For left tails, a positive-probability event fixes a subset of the binary coordinates at $0$.  Increasing the left-tail threshold removes such fixations.  The same exact comparison $\mu_0\ge_{\rm st}\mu_1$ shows that the distribution after imposing $X_j=0$ stochastically dominates the mixture before that fixation.  As above, if one exact cell has probability zero, the conclusion is either equality or the stronger tail event is inadmissible.  Iterating proves NLTD.
\end{proof}

\subsection{An exact \texorpdfstring{$2\times2\times3$}{2 x 2 x 3} counterexample}
\label{sec:counterexample}

Let $\bm X=(X_1, X_2, X_3)$ be a random vector such that $X_1, X_2\in\{0,1\}$ and $X_3\in\{0,1,2\}$.  Define its joint distribution by dividing every entry of Table \ref{tab:weights-nrd} by $10^8$.

\begin{table}[ht]
\centering
\caption{Integer weights for the counterexample.  All $12$ entries are positive and sum to $10^8$.} \medskip 
\begin{tabular}{ccrrr}   \hline 
   $X_1$ & $X_2$ & $X_3=0$ & $X_3=1$ & $X_3=2$ \\ \hline 
   $0$ & $0$ & $7$          & $50\,944$   & $2\,832\,845$ \\
   $0$ & $1$ & $8\,799\,436$ & $164\,145$  & $164\,145$ \\
   $1$ & $0$ & $20\,365\,006$& $379\,890$  & $379\,890$ \\
   $1$ & $1$ & $65\,617\,650$& $1\,224\,030$ & $22\,012$ \\  \hline 
\end{tabular}
\label{tab:weights-nrd}
\end{table}

\begin{proposition}\label{prop:counterexample}
The distribution in Table {\rm \ref{tab:weights-nrd}} is full {\rm NRD} but is not {\rm NRTD}.
\end{proposition}

\begin{proof}
Apply Lemma \ref{lem:certificate}. Enumerating all ordered disjoint index pairs $(I,J)$ and all nontrivial upper sets yields 56 cover inequalities. Table \ref{tab:certificate} organizes these inequalities by $(I,J)$ and reports the minimal exact cross-product slack within each group.  Every listed minimum is strictly positive, with an overall minimum of $59\,730$.  It follows that all $56$ inequalities are satisfied, so the distribution is full NRD.

\begin{table}[ht]
\centering
\caption{Exact full-NRD certificate.  The labels $0,1,2$ denote $X_1, X_2, X_3$, respectively.} \medskip 
\begin{tabular}{crr@{\qquad}crr}  \hline 
   $I\mid J$ & count & minimum slack & $I\mid J$ & count & minimum slack \\   \hline 
   $0\mid1$  & 1 & $10\,881\,804$       & $0\mid2$  & 2 & $4\,380\,320\,879\,824$ \\
   $0\mid12$ & 7 & $59\,730$            & $1\mid0$  & 1 & $10\,881\,804$ \\
   $1\mid2$  & 2 & $3\,791\,208\,908\,449$ & $1\mid02$ & 7 & $59\,730$ \\
   $2\mid0$  & 2 & $258\,532\,967\,344\,678$ & $2\mid1$  & 2 & $239\,082\,495\,973\,618$ \\
   $2\mid01$ & 8 & $10\,577\,151\,486\,188$ & $01\mid2$ & 8 & $3\,342\,963\,269\,880$ \\
   $02\mid1$ & 8 & $10\,881\,804$       & $12\mid0$ & 8 & $10\,881\,804$ \\  \hline 
\end{tabular}
\label{tab:certificate}
\end{table}

For the right-tail comparison, direct summation gives
\begin{align}
    \P (X_1=1\,|\, X_2\ge 0, X_3\ge 1) &=\frac{2\,005\,822}{5\,217\,901},    \label{eq:broad-tail}\\[3pt]
    \P (X_1=1\,|\, X_2\ge 1, X_3\ge 1) &=\frac{1\,246\,042}{1\,574\,332}.    \label{eq:narrow-tail}
\end{align}
Then $ \P (X_1=1\,|\, X_2\ge 0, X_3\ge 1)<  \P (X_1=1\,|\, X_2\ge 1, X_3\ge 1)$ since
\begin{equation}\label{eq:reversal}
  2\,005\,822\times 1\,574\,332  -1\,246\,042\times 5\,217\,901  =-3\,343\,894\,036\,938 < 0.
\end{equation}
In the strict-threshold notation of \eqref{eq:nrtd}, the two conditioning thresholds for $(X_2, X_3)$ are $(-1,0)\le (0,0)$.  Taking the increasing test function $f(X_1)=X_1$ proves that the distribution is not NRTD.

The enumeration in Table \ref{tab:certificate}, including an additional check of all $74$ upper-set inequalities arising from comparable conditioning-value pairs, is reproduced by the exact-arithmetic ancillary file 
\begin{center}
      \path{anc/nrd-tail-counterexample/verify_counterexample.py}.  
\end{center}
\noindent The script uses only the Python standard library and asserts every count and displayed integer.
\end{proof}

\begin{remark} {\rm (The mixing mechanism)\ \ 
The reversal in \eqref{eq:reversal} is a finite version of Simpson's phenomenon.  Define
$$
  q_{bc}=\P (X_1=1\,|\, X_2=b, X_3=c).
$$
It is easy to see that
\begin{align*}
 q_{01}&=\frac{379\,890}{430\,834}\approx0.88175492
       \ge \frac{1\,224\,030}{1\,388\,175}\approx0.88175482=q_{11},\\[3pt]
 q_{02}&=\frac{379\,890}{3\,212\,735}\approx0.11824505
       \ge \frac{22\,012}{186\,157}\approx0.11824428=q_{12}.
\end{align*}
However, when conditioning on $X_3\ge1$, the weight of the high-$q$ layer $X_3=1$ satisfies
$$
    \P (X_3=1\,|\, X_2=0, X_3\ge 1) =\frac{430\,834}{3\,643\,569}  \approx0.11824505,
$$
whereas
$$
    \P (X_3=1\,|\, X_2=1, X_3\ge 1)=\frac{1\,388\,175}{1\,574\,332}  \approx0.88175493.
$$
Consequently,
\begin{align*}
   \P (X_1=1 | X_2=0, X_3\ge 1) & =q_{01} \P(X_3=1 | X_2=0, X_3\ge 1) +q_{02} \P(X_3=2 | X_2=0, X_3\ge 1)\\
   & \approx 0.20852631  <0.79147346\approx \P (X_1=1\,|\, X_2=1, X_3\ge 1 ).
\end{align*}
This is the NRTD violation in \eqref{eq:broad-tail}--\eqref{eq:reversal}.  A two-point coordinate has no two-layer proper upper tail of this kind: every nontrivial upper tail is a singleton.  
}
\end{remark}

\subsection{Reflection and minimality}

\begin{proposition}\label{prop:reflection}
Let $\bm X= (X_1, X_2, X_3)$ have the distribution in Table {\rm \ref{tab:weights-nrd}}, and set $\bm Y=- \bm X$.  Then $\bm Y$ is full {\rm NRD} but is not {\rm NLTD}.
\end{proposition}

\begin{proof}
Simultaneous coordinate reflection preserves full NRD.  Indeed, if $y_J\le y_J^*$, then $-y_J^*\le-y_J$.  Applying NRD to $\bm X$ and then using the fact that negation reverses the usual stochastic order gives
$$
     [Y_I\,|\, Y_J=y_J^\ast ]\le_{\rm st} [Y_I\,|\, Y_J=y_J].
$$
For the left-tail failure, \eqref{eq:broad-tail}--\eqref{eq:reversal} give
\begin{align*}
 \E [Y_1\,|\, Y_2\le -1, Y_3\le -1]   &=-\frac{1\,246\,042}{1\,574\,332}\\[3pt] 
   & < -\frac {2\,005\,822}{5\,217\,901}     =\E [Y_1\,|\, Y_2\le 0, Y_3\le -1].
\end{align*}
The threshold vectors satisfy $(-1,-1)\le(0,-1)$, while $Y_1$ is an increasing function of the output coordinate.  This contradicts \eqref{eq:nltd}.
\end{proof}

\begin{proposition}\label{cor:minimality}
If a finite full-{\rm NRD} random vector violates  {\rm NRTD} or {\rm NLTD}, then every witnessing projection has rectangular size at least $12$.  This bound is achieved on a $2\times2\times3$ state space.
\end{proposition}

\begin{proof}
Let disjoint nonempty blocks $I, J$ witness a failure, and project to $\bm X_{I\cup J}$.  Delete every one-point coordinate, which neither changes the conditioning events nor the conditional output distributions.  If the remaining
conditioning block had only one coordinate, Proposition \ref{prop:singleton-tail} would exclude the failure.  It therefore has at least two coordinates.  By Proposition \ref{thm:binary-block}, at least one of them has at least three marginal states, and another has at least two.

The joint output support of $\bm X_I$ has at least two points, since a one-point conditional distribution cannot violate stochastic order.  Hence the product of the effective marginal support cardinalities over $I$ is at least two.  The
rectangular size of the witnessing projection is therefore at least $2\times 2\times 3=12$. The distribution in Table \ref{tab:weights-nrd} attains this bound for NRTD.
\end{proof}

The minimality statement in Proposition \ref{cor:minimality} refers to rectangular marginal state-space cardinalities, not the minimum number of positive atoms inside a fixed rectangle domain.  The displayed counterexample is strictly positive and uses all $12$ atoms.  No assertion is made here about the sparsest counterexample inside a $2\times2\times3$ grid.

\section{Full NRD does not imply NA}
\label{sec:counterexample-na}

In this section, we show that there exists a random vector $\bm X=(X_1,X_2,X_3,X_4)$ with a strictly positive
mass function on $\{0,1,2\}^4$ such that $\bm X$ is full NRD but is not NA. We also show that full NRD does imply NA for at most three coordinates and, in the four-coordinate case, whenever one coordinate is binary.

\subsection{The exact counterexample}
\label{sect-na-1}
 
Index the rows and columns of Table \ref{tab:weights}, in the displayed order, by the two-coordinate states
$00, 01, 02, 10$, $11, 12, 20, 21, 22$. The row is $(X_1,X_2)$ and the column is $(X_3,X_4)$.  Define the mass of an atom to be its table entry $w(\bm x)$ divided by $99\,999$.  All $81$ entries are positive, and direct summation gives
\begin{equation}
 \label{eq:total-mass}
     \sum_{\bm x\in\{0,1,2\}^4} w(\bm x)=99\,999.
\end{equation}
 
\begin{table}[ht]
\centering
\small 
\setlength{\tabcolsep}{6.5pt}
\caption{Positive integer weights for the full-NRD, non-NA distribution. The rows are indexed by $(X_1,X_2)$ and the columns by $(X_3,X_4)$.}   \vskip 8pt 
\begin{tabular}{c@{\quad}rrrrrrrrr}      \hline 
\quad  & $00$ &  $01$ & $02$ & $10$ & $11$ & $12$ & $20$ & $21$ & $22$ \\    \hline 
   $00$ &    1 &   93 &  959 &   93 & 1708 &  397 &  959 & 397 & 209 \\
   $01$ &   93 &  572 & 1783 &  572 & 2247 &  450 & 1783 & 450 & 238 \\
   $02$ &  959 & 1783 & 3918 & 1783 & 5065 &  857 & 3918 & 857 & 439 \\
   $10$ &   93 &  572 & 1783 &  572 & 2247 &  450 & 1783 & 450 & 238 \\
   $11$ & 1708 & 2247 & 5065 & 2247 & 8193 & 1275 & 5065 & 1275 & 659 \\
   $12$ &  397 &  450 &  857 &  450 & 1275 &  187 &  857 & 187 &  93 \\
   $20$ &  959 & 1783 & 3918 & 1783 & 5065 &  857 & 3918 & 857 & 439 \\
   $21$ &  397 &  450 &  857 &  450 & 1275 &  187 &  857 & 187 &  93 \\
   $22$ &  209 &  238 &  439 &  238 &  659 &   93 &  439 &  93 &   1 \\   \hline 
\end{tabular}
\label{tab:weights}
\end{table}
 
We next describe the complete NRD audit.  The numbers of upper sets of $\{0,1,2\}^k$, including the empty and full sets, are 
\begin{equation}
  \label{eq:upper-counts}
       u_1=4,\qquad u_2=20,\qquad u_3=980.
\end{equation}
The enumeration is recursive: an upper set in $\{0,1,2\}^k$ is specified by three upper sets $U_0\subseteq U_1\subseteq U_2$ in $\{0,1,2\}^{k-1}$, one for each value of the last coordinate.  The number of strict comparable pairs in $\{0, 1, 2\}^k$ is
\begin{equation}
   \label{eq:comparable-counts}
         c_k= 6^k-3^k,\qquad (c_1, c_2, c_3)=(3, 27, 189),
\end{equation}
because one coordinate admits six ordered pairs $(a, b)$ with $a\le b$, of which three are equal.

For block sizes $i=|I|$ and $j=|J|$, the number of slacks in
\eqref{eq:slack} is
\[
  \binom{4}{i}\binom{4-i}{j}(u_i-2)c_j.
\]
The complete breakdown is shown in Table \ref{tab:certificate-counts}.

\setcounter{table}{1}

\begin{table}[ht]
\centering
\caption{Number of exact full-NRD inequalities, grouped by block sizes.} \vskip 8pt
\begin{tabular}{ccc@{\qquad}ccc}    \hline 
   $|I|$ & $|J|$ & count & $|I|$ & $|J|$ & count \\    \hline 
   $1$ & $1$ & $72$    & $1$ & $2$ & $648$ \\
   $1$ & $3$ & $1\,512$ & $2$ & $1$ & $648$ \\
   $2$ & $2$ & $2\,916$ & $3$ & $1$ & $11\,736$ \\   \hline 
   \multicolumn{5}{r}{Total} & $17\,532$ \\  \hline 
\end{tabular}
\label{tab:certificate-counts}
\end{table}

\begin{proposition}\label{prop:strict-nrd}
The distribution in Table {\rm\ref{tab:weights}} is full {\rm NRD}.  More precisely, all $17\,532$ slacks in \eqref{eq:slack} are positive, and their minimum is $3\,958$.
\end{proposition}

\begin{proof}
The exact-arithmetic ancillary verifier is
\begin{center}
  \path{anc/full-nrd-not-na/verify.py}.
\end{center}
It generates the upper sets by the recursion preceding \eqref{eq:comparable-counts}, loops over every ordered pair of disjoint nonempty blocks, and evaluates \eqref{eq:slack} for every nontrivial upper set and strict comparable conditioning pair.  Its loop counts agree with Table \ref{tab:certificate-counts}.  All operands are integers.

The smallest slack in \eqref{eq:slack} occurs for $I=\{1\}$, $J=\{2,3,4\}$, $\bm a=(0,2,2)$, $\bm b=(1,2,2)$, and $U=\{x_1\ge1\}$. For this comparison the two conditional numerator-denominator pairs are $677/886$ and $752/990$, so the exact cross product is
\begin{equation*}
 \label{eq:minimum-slack}
  677\times 990 -752\times 886 = 3\,958 >0.
\end{equation*}
The script asserts both this value and the total count $17\,532$, and aborts on any nonpositive slack.  Hence every inequality required by Lemma \ref{lem:certificate} holds, which proves full NRD.
\end{proof}

The failure of NA is visible in one exact calculation.

\begin{proposition}\label{prop:na-failure}
The distribution in Table {\rm\ref{tab:weights}} is not {\rm NA}.
\end{proposition}

\begin{proof}
Consider the increasing events on disjoint blocks $A=\{X_1\ge1,\ X_2\ge1\}$ and $B=\{X_3\ge1,\ X_4\ge1\}$. Summing the corresponding entries in Table \ref{tab:weights} yields  $w(A)=w(B)=39\,649$ and $A\cap B)=15\,732$. Together with \eqref{eq:total-mass}, this gives
\begin{align}
\label{eq:positive-covariance}
  99\,999\,w(A\cap B)-w(A)w(B) &=99\,999\times 15\,732-39\,649^2 =1\,141\,067>0.
\end{align}
After division by $99\,999^2=9\,999\,800\,001$, \eqref{eq:positive-covariance} is precisely 
$$
   \mathrm{Cov} (\id_A, \id_B)=\frac {1\,141\,067}{9\,999\,800\,001} >0, 
$$
contradicting \eqref{eq:na}.
\end{proof}

\begin{remark}\label{rem:reproducibility}
{\rm 
The expanded weights are also stored in the ancillary file
\begin{center}
  \path{anc/full-nrd-not-na/weights.txt}.
\end{center}
Its SHA-256 digest is the concatenation of the following two lines:
\begin{center}
\texttt{cf4bca92dfcbc3a899fc040d930919b1}\\
\texttt{927520045731c445ddb02d844e5623b1}.
\end{center}
Two independent Python implementations and an independent C++ implementation using \texttt{\_\_int128} arithmetic are included in the same ancillary directory.  An exhaustive audit of all $9\,252$ upper-event pairs on disjoint blocks finds exactly one positive covariance and no zero covariance; the positive pair is the pair in Proposition \ref{prop:na-failure}. This uniqueness is not needed for the counterexample.  }
\end{remark}

\subsection{Why smaller coordinate structures are safe}

The counterexample in Subsection \ref{sect-na-1} is four-dimensional. Does such a counterexample exist for a three-dimensional random vector? Proposition \ref{prop:three-coordinates} below gives a negative answer.

\begin{proposition}
 \label{prop:three-coordinates}
A three-dimensional random vector $\bm X$ with full {\rm NRD} property is {\rm NA}.
\end{proposition}

\begin{proof}
Fix disjoint nonempty blocks $I,J$ and increasing functions $f$ and $g$. Since $|I|+|J|\le3$, one of the blocks is a singleton. By symmetry of covariance, suppose $I=\{i\}$, and denote $h(t)=\E [g(\bm X_J)\,|\, X_i=t]$. Full NRD makes $h$ decreasing.  By the tower property and Chebychev's inequality, we have
$$
  \mathrm{Cov} (f(X_i),g(X_J)) = \mathrm{Cov} (f(X_i),h(X_i))\le 0.
$$
This proves \eqref{eq:na} for every pair of blocks.
\end{proof}

The four-dimensional case reduces to a covariance lemma for two random variables, one of which is a Bernoulli random variable. 

\begin{lemma}
 \label{lem:binary-chain}
Let $(Y,Z)$ be a two-dimensional random vector with $\P(Y=1)=p=1-\P(Y=0)$. Suppose the function $\eta:\R^2\mapsto \R$ satisfies the following:
\begin{itemize}
    \item[{\rm (1)}] for each fixed $y$, the mapping $z \mapsto \eta(y, z)$ is decreasing;
    \item[{\rm (2)}] $\E [\eta(Y, Z)\,|\, Y = 0] \ge \E [\eta(Y, Z)\,|\, Y = 1]$;
    \item[{\rm (3)}] the function $z \mapsto \E [\eta (Y, Z)\,|\, Z = z]$ is decreasing.
\end{itemize}
Then $\mathrm{Cov} (\eta(Y,Z), g(Y,Z)) \le 0$ for every increasing real-valued function $g$.
\end{lemma}

\begin{proof} 
Define $g (y,z) = \E [\eta(Y,Z)] - \eta(y, z)$. Then $g(y,z)$ is increasing in $z$ for each $y$.  Assumption (2), combined with the identity $\E [g(Y,Z)]= 0$, implies
$$
   \E [g(Y,Z)\, |\, Y = 1] =p\, \E [\eta(Y, Z)\,|\, Y =1] + (1-p)\,\E [\eta(Y, Z)\,|\, Y = 0] -\E [\eta(Y, Z)\,|\, Y =1]   \ge 0,
$$
or, equivalently, $\E \left [g(Y,Z)\id_{\{Y=1\}}\right ]\ge 0$. Consequently, 
\begin{equation}
\label{eq-260915}
    \E \left [g(Y,Z)\id_{\{Y=1,Z\ge a\}}\right ]\ge 0,\quad a\in\R,
\end{equation}
since $g(y,z)$ is inceasing in $z$ for each $y$. By assumption (3), it follows that
$$
    \E [g(Y,Z)\,|\, Z = z] = \E [\eta(Y,Z)]- \E [\eta(Y, Z)\,|\, Z = z]
$$
is increasing in $z$, and $\E [g(Y,Z)\,|\, Z] $ has mean zero. Hence, $\E [g(Y,Z)\,|\, Z\ge a]\ge 0$ and $\E \left [g(Y,Z) \id_{\{Z\ge a\}}\right ]\ge 0$ for all $a\in \R$.

Every upper set of $\{0, 1\} \times \R$ has the form
$$
    U_{a,b} = \{Y = 1, Z \ge a\} \cup \{Y = 0, Z \ge b\}, \quad -\infty\le a \le b\le \infty.
$$
Thus,
$$
    \E \left [g(Y,Z) \id_{U_{a,b}}\right ] =\E \left [g(Y,Z)\id_{\{Y=1,Z\ge a\}}\right ]+ \E\left [g(Y,Z)\id_{\{Y=0,Z\ge b\}}\right ].
$$
Two cases arise.
\begin{itemize}
  \item If $\E \left [g(Y,Z)\id_{\{Y=0,Z\ge b\}}\right ] \ge 0$, it follows that $\E \left [g(Y,Z) \id_{U_{a,b}}\right ]\ge 0$ in view of \eqref{eq-260915}. 
      
  \item If $\E\left [g(Y,Z)\id_{\{Y=0,Z\ge b\}}\right ] < 0$, by the monotonicity of $g(y,z)$ in $z$ for each $y$, we obtain 
     $$
            \E [g(Y,Z)\,|\, Y=0, a\le Z< b] \le \E [g(Y,Z)\,|\, Y=0,Z\ge b] <0,
     $$
     implying $\E\left [g(Y,Z)\id_{\{Y=0,a\le Z< b\}}\right ]<0$. Thus,
     $$
         \E \left [g(Y,Z) \id_{U_{a,b}}\right ] = \E \left [g(Y,Z) \id_{\{Z\ge a\}}\right ] - \E\left [g(Y,Z) \id_{\{Y=0, a\le Z< b\}}\right ] \ge \E \left [g(Y,Z) \id_{\{Z\ge a\}}\right ] \ge 0.
     $$
\end{itemize}
Therefore, $\operatorname{Cov}(\eta(Y,Z), \id_U) = -\E [g(Y,Z) \id_U] \le 0$ for every upper set $U$. On a finite poset, every increasing function is, after subtraction of a constant, a non-negative linear combination of upper-set indicators. The claimed inequality follows.
\end{proof}

\begin{proposition}\label{prop:four-binary}
Let $\bm X=(X_1,X_2,X_3,X_4)$ have a strictly positive mass function and be full {\rm NRD}.  If one coordinate $X_i$ is binary, then $\bm X$ is {\rm NA}.
\end{proposition}

\begin{proof}
Coordinate projection preserves full NRD.  Thus every pair of disjoint blocks using at most three coordinates is covered by Proposition \ref{prop:three-coordinates}.  It remains to consider partitions of all four coordinates.  A singleton-versus-three-coordinate partition is handled by the proof of Proposition \ref{prop:three-coordinates}, so only a two-versus-two partition needs attention.

For such a partition, let $\bm X_{1,2}=(X_1,X_2)$ be the side containing the binary coordinate, say $X_1$, and let $\bm X_{3,4}=(X_3, X_4)$ be another side.  By symmetry of covariance, it is enough to consider an increasing function $f$ of $\bm X_{3,4}$ and an increasing function $g$ of $\bm X_{1,2}$.  Set
$$
  \eta(y,z) = \E [f(\bm X_{3,4})\,|\, X_1=y, X_2=z].
$$
Full NRD with conditioning block $\bm X_{1,2}$ shows that $\eta(y,z)$ is decreasing in $z$ for each fixed $y$.  Full NRD with conditioning block $X_1$ gives
$$
  \E [\eta(\bm X_{1,2})\,|\, X_1=0]\, \ge\, \E [\eta (\bm X_{1,2})\,|\, X_1=1],
$$
and conditioning on $X_2$ shows that $\E [\eta(\bm X_{1,2})\,|\, X_2=z]$ is decreasing in $z$.  Hence, Lemma \ref{lem:binary-chain} applies, and the tower property yields
$$
  \mathrm{Cov} (f(\bm X_{3,4}), g(\bm X_{1,2})) = \mathrm{Cov} (\eta(\bm X_{1,2}), g(\bm X_{1,2}))\le 0.
$$
This proves NA for the last remaining type of block pair.
\end{proof}

\begin{remark}\label{rem:minimality-scope}
{\rm 
By Propositions \ref{prop:three-coordinates} and \ref{prop:four-binary}, any strictly positive four-coordinate counterexample must contain at least three states in each coordinate. The state space $\{0,1,2\}^4$ for the counterexample in Subsection \ref{sect-na-1} is coordinate-wise minimal among four-coordinate rectangular state spaces. This, 
however, does not imply that $81$ is the minimal total rectangular size for counterexamples with an arbitrary number of coordinates, nor does it identify the minimal number of positive atoms when zero masses are allowed.
}
\end{remark}

\section*{Funding}

T. Hu, the corresponding author,  would like to acknowledge financial support from National Natural Science Foundation of China (No. 72332007, 12371476).  

\section*{Acknowledgement}
	
The authors thank Tianyang Sun of the School of Mathematics Sciences, University of Science and Technology of China, for using the Eureka system to produce a preliminary version of the proof. Eureka is a multi-agent system for resolving mathematical conjectures through human-AI interaction. 

\section*{Disclosure statement}

No potential conflict of interest was reported by the authors.

\end{document}